\documentclass[10pt,a4paper]{amsart}
\usepackage{verbatim}

\usepackage[T1]{fontenc}
\usepackage{graphicx}
\usepackage{enumerate}
\usepackage{amsmath,amsfonts,amssymb}
\usepackage{color}
\usepackage{cite}
\definecolor{citation}{rgb}{0.2,0.58,0.2} 
\definecolor{formula}{rgb}{0.1,0.2,0.6}
\definecolor{url}{rgb}{0.3,0,0.5}
\usepackage{pgf,tikz}
\usepackage{mathrsfs}
\usepackage{fancyhdr}
\usepackage{dutchcal}

\usepackage{dsfont}

\newcommand{\reqnomode}{\tagsleft@false}

\newcommand\R{\mathbb{R}}

\allowdisplaybreaks
\makeatletter
\DeclareRobustCommand*{\bfseries}{%
  \not@math@alphabet\bfseries\mathbf
  \fontseries\bfdefault\selectfont
  \boldmath
}

\makeatother

\newlength{\defbaselineskip}
\newtheorem{theorem}{Theorem}
\newtheorem{corollary}{Corollary}[section]
\newtheorem{definition}{Definition}
\newtheorem{remark}{Remark}[section]
\newtheorem{lemma}{Lemma}[section]

\numberwithin{equation}{section}
\allowdisplaybreaks[4]

\newcommand\eps\varepsilon
\def\eqn#1$$#2$${\begin{equation}\label#1#2\end{equation}}

\newcommand{\be}{\begin{equation}}
\newcommand{\ee}{\end{equation}}

\newcommand{\G}{\mathcal{G}}

\newcommand{\dv}{\operatorname{div}}

\newcommand{\RN}{{\mathbb{R}^{N}}}
\newcommand{\rn}{{\RN}}

\newcommand{\rpp}{{[0,\infty)}}
\def\name[#1, #2]{#1 #2}

\newcommand{\opA}{{\mathcal{A}}}
\newcommand{\V}{{\mathcal{V}_G}}
\newcommand{\wt}{\widetilde}

\begin{document}
\title[Controlling monotonicity of nonlinear operators]{Controlling monotonicity of nonlinear operators}

\author[Borowski]{Micha\l{} Borowski}   \address{Micha\l{} Borowski\\Faculty of Mathematics, Informatics and Mechanics, University of Warsaw\\ul. Banacha 2, 02-097 Warsaw, Poland}\email{\texttt{m.borowski@mimuw.edu.pl}} 
\author[Chlebicka]{Iwona Chlebicka} \address{Iwona Chlebicka\\Faculty of Mathematics, Informatics and Mechanics, University of Warsaw\\ul. Banacha 2, 02-097 Warsaw, Poland} \email{\texttt{i.chlebicka@mimuw.edu.pl}}

\subjclass[2020]{26D20, 35A25\vspace{1mm}} 

\keywords{convex functions, $p$-Laplacian, monotonicity of operators, operators of nonstandard growth, vectorial inequalities\vspace{1mm}}

\thanks{{\it Acknowledgements.}\ {M.~B. is supported by the Ministry of Science and Higher Education project Szko\l{}a Or\l{}\'o{}w, project number 500-D110-06-0465160, I.~C. is supported by NCN grant 2019/34/E/ST1/00120.} 
\vspace{1mm}}

\maketitle

\begin{abstract} Controlling the monotonicity and growth of Leray--Lions' operators including the $p$-Laplacian plays a fundamental role in the theory of existence and regularity of solutions to second order nonlinear PDE. We collect, correct, and supply known estimates including the discussion on the constants. Moreover, we provide a comprehensive treatment  of related results for operators with Orlicz growth. We pay special attention to exposition of the proofs and the use of elementary arguments.
\end{abstract} 

\section{Introduction} One of the most important branches in the modern nonlinear analysis is a study of the properties of solutions to PDE involving Leray--Lions' monotone operators~\cite{LL} with a preeminent example of the   $p$-Laplace operator
\[-\Delta_p u=\dv(|Du|^{p-2}Du),\quad 1<p<\infty,\]
and its weighted counterpart $-\Delta_p^a u=\dv(a(x)|Du|^{p-2}Du)$ with $0\ll a\in L^\infty$.  The ellipticity of such operators depends on the value of a parameter $p$. When $p=2$, then the operator  $\Delta_p=\Delta$ is the classical Laplacian.  Moreover, the $p$-Laplace operator is called degenerate when $p>2$ and singular when $1<p<2$. Typically different methods are applied in these cases \cite{Lindqvist}. This is a~major obstacle to be overcome in the study of the operators of general growth such as~\eqref{eq:A}.

Controlling the monotonicity and the growth of operators like $\Delta_p$ or $\Delta_p^a$ has a preeminent place in the theory of existence, regularity, and other properties of solutions to nonlinear PDE involving these operators. To give a flavour of a vast literature where such tools are applied we refer to~\cite{LL,BM,MT,Db,Db2,hekima,KM,Min2007,simon,Uhl}. We point out that in a study of regularity of minimizers to a related variational problem 
\[\inf_{u\in u_0+W^{1,p}_{loc}}\int |Du|^p\,dx\] similar estimates are frequently employed as well, see e.g.~\cite{giaquinta,hamburger}. Going into more detail, a fundamental tool in the study of the properties of solutions to the mentioned nonlinear PDE involving Leray--Lions' operators and minimizers is a precise control over the monotonicity rate described by the quantity
\begin{equation}
    \label{mono-p}
\mathcal{J}_p(\xi,\eta):=\left\langle |\xi|^{p-2}\xi-|\eta|^{p-2}\eta,\xi-\eta\right\rangle\qquad\text{for }\ \xi,\eta\in\rn,
\end{equation}
where $\langle \cdot,\cdot\rangle$  denotes a scalar product. The quantity $\mathcal{J}_p$ appears when one subtracts two differential problems involving the $p$-Laplacian tested against the difference of the solutions. This procedure is typical in the regularity theory. It can be applied in order to infer regularity of solutions to problems with less regular data by comparison with solutions to problems involving more regular data, see~\cite[Section~4]{Min2007}. Another  example is a perturbation method based on a comparison between a~solution to a weighted problem and a solution to a related unweighted one, see  \cite[Theorem~3.8]{HanLin}.

In the further parts of the paper, we prove a variety of lemmas providing estimates of $\mathcal{J}_p(\xi,\eta)$.  We are inspired by the proofs from~\cite{BIV,Db,Db2,Lindqvist,simon}, which are here supplied with details or fixed when necessary. It is typical to provide different methods of the proof of estimates of $\mathcal{J}_p(\xi,\eta)$ for $p<2$ and $p>2$, see Lemmas~\ref{lem:p-male} and \ref{lem:p>2} here, and \cite{Lindqvist}. It can be particularly useful to have one function to control the monotonicity in the full range $1<p<\infty$. For this purpose we make use of the following auxiliary functions\begin{equation}\label{Vp}
    \bar V_p^\mu(\xi):=(\mu+|\xi|^2)^\frac{p-2}{4}\xi,\ \mu\geq0, \qquad\text{ and }\qquad V_p(\xi):=|\xi|^\frac{p-2}{2}\xi.
\end{equation}
Clearly, $\bar V_p^0(\xi)=V_p(\xi)$. Such functions are used to express differentiability properties of gradients of solutions to problems like $-\Delta_p u=0$, cf.~\cite{Min2007,Uhl}. 

Following the pioneering  contributions~\cite{Ta79,Go74,Li} and some more recent ones~\cite{baroni,CGZG-Wolff,CYZG,CiFu,CiMa,CiMa2,DiEt,DiLeStVe} we are interested in a class of operators embracing $\Delta_p$ given by
\begin{equation}
    \label{eq:A}\mathfrak{A}[u]:=-\dv \opA (x,Du)=-\dv\left(a(x)\frac{G(|Du|)}{|Du|^2}Du\right),
\end{equation} 
where $G$ is a certain $N$-function, which is similar to a power function in the sense that $G(0)=0$, $G$ is convex,  and satisfies growth conditions of doubling type called $\Delta_2$ and $\nabla_2$-condition (cf. Definition~\ref{def:D2}). As for regularity of $G$, we need only $G\in C(\rpp)\cap C^1((0,\infty))$ and we denote its derivative  $g=G'$. For the choice of $G(t)=t^p$ and $a\equiv 1$ we are back in the $p$-Laplace case. Apart from the power functions, the typical examples are Zygmund functions $G(t)=t^p\log^\alpha(e+t)$, $p>1$ and $\alpha\in\R$ together with their compositions and multiplications of their compositions. Functions satisfying both $\Delta_2$ and $\nabla_2$-conditions are trapped between power functions with exponents $i_G$ and $s_G$, called indices, see Lemma~\ref{lem:iG-sG}. For more information on convex functions, growth conditions, and Orlicz spaces we refer to~\cite{KR,rao-ren}. In order to analyze solutions of PDE involving operator $\mathfrak{A}$ from~\eqref{eq:A}, one may need to estimate the following expression describing monotonicity 
\begin{equation}
    \label{mono-G} \mathcal{J}_G:=\left\langle \frac{G(|\xi|)}{|\xi|^2}\xi-\frac{G(|\eta|)}{|\eta|^2}\eta,\xi-\eta\right\rangle \qquad\text{for $\xi,\eta\in\rn$ and $\xi,\eta\neq 0$.} \end{equation}
 Note that \eqref{mono-G} is a general growth version of \eqref{mono-p}. One of the ideas to by-pass the challenge resulting from the fact that in the Orlicz case there is no natural threshold like $p=2$ in the power growth case is to prove estimates on $\mathcal{J}_G$ with the use of a few different auxiliary functions, see e.g.~\cite{DiEt}. In a direct relation to functions from \eqref{Vp} we define \begin{equation}
    \label{V}\V(\xi):=\frac{\sqrt{G(|\xi|)}}{|\xi|}\xi\qquad\text{for $\xi\in\rn$ and $\xi\neq 0$.}
\end{equation} 
In the case of $G(t)=t^p$, we consider $V_p$ and $\bar V_p^\mu$ given by~\eqref{Vp}. Then we have $V_p(\xi)=\bar{V}_p^0(\xi)=\V(\xi).$ Moreover, let\begin{equation}
    \label{G-shift}G_a(t):=\int_0^t \frac{g(a+s)}{a+s}s\,ds \qquad\text{for $a,t>0$ }
\end{equation}
and \begin{equation}
    \label{cG}\G(\xi,\eta):=\begin{cases}G(|\xi-\eta|),&|\xi-\eta|\geq |\xi|,\\
    \tfrac{G(|\xi|)}{|\xi|^2}|\xi-\eta|^2,& |\xi-\eta|\leq |\xi|,
    \end{cases} \qquad\text{for $\xi,\eta\in\rn$ and $\xi\neq 0$.}
\end{equation}  
Our aim is to provide upper and lower bound for $\mathcal{J}_G$ by the same quantities involving $\V$, $G_a$, and $\G$. The estimates shall hold up to a multiplicative constant dependent only on indices $i_G$ and $s_G$. We denote  by $a\lesssim b$ that a function $a$ is dominated by a function $b$ up to a constant. We write $a\approx b$, if $a\lesssim b$ and $b\lesssim a$. In order to stress dependence of the intrinsic constant on certain parameters we may abbreviate the notation by placing them in the lower index, e.g. $\lesssim_{i_G}$ or $\approx_p$. Some of the estimates were already known, see e.g.~\cite{DiEt,DiLeStVe}. More precise references are given as comments to particular lemmas. We provide new elementary proofs not requiring $G$ to be twice differentiable. Our main result reads as follows.

\begin{theorem}\label{theo:main}
If $G\in C(\rpp)\cap C^1((0,\infty))$ is an $N$-function such that $G\in\Delta_2\cap\nabla_2$, $g=G'$, $\V$ is given by~\eqref{V},  $G_a$ by~\eqref{G-shift},  $\G$ by~\eqref{cG}, then we have
\begin{equation}
    \begin{split}\label{equiv-theo:main}
 |\V(\xi)-\V(\eta)|^2&\approx_{i_G, s_G} G_{|\xi|}(|\xi-\eta|)  \approx_{i_G, s_G} \G(\xi,\eta)\approx_{i_G, s_G} \tfrac{g(|\xi| + |\eta|)}{|\xi| + |\eta|}|\xi - \eta|^2\\
&\approx_{i_G, s_G} \left\langle \tfrac{G(|\xi|)}{|\xi|^2}\xi-\tfrac{G(|\eta|)}{|\eta|^2}\eta,\xi-\eta\right\rangle=\mathcal{J}_G(\xi,\eta).
\end{split}
\end{equation}
\end{theorem}
Putting $G(t) = t^p$ in Theorem~\ref{theo:main} immediately gives us the following corollary.
\begin{corollary}\label{coro:main} 
Suppose  $1<p<\infty$, $V_p$ is given by~\eqref{Vp}, and $\mathcal{J}_p$ by~\eqref{mono-p}. Then 
\begin{equation*}
|V_p(\xi)-V_p(\eta)|^2  \approx_p \big(|\xi|^2 + |\eta|^2\big)^{\frac{p-2}{2}}|\xi-\eta|^2\approx_{p} \left(|\xi| + |\eta|\right)^{p-2}|\xi - \eta|^2 \approx_{p} \mathcal{J}_p(\xi,\eta).
\end{equation*}
\end{corollary}
Let us point out that actually Corollary~\ref{coro:main} results from Lemmas~\ref{lem:pow-approx} and~\ref{lem:Vp}, which are proven in Section~\ref{sec:p} by elementary methods adequate for the power growth operators. Later on, in Section~\ref{sec:G}, we use these lemmas to prove Theorem~\ref{theo:main}.  

\medskip

If a function $G$ is superquadratic, the lower bound in Theorem~\ref{theo:main} gets simplified.
\begin{corollary}
If $G\in C(\rpp)\cap C^1((0,\infty))$ is an $N$-function such that $G\in\Delta_2\cap\nabla_2$ which is superquadratic (i.e. $i_G\geq 2$) and $g=G'$,  then \begin{equation}
    \label{eq:coro}
G(|\xi-\eta|)\lesssim_{i_G,s_G}
\left\langle \tfrac{G(|\xi|)}{|\xi|^2}\xi-\tfrac{G(|\eta|)}{|\eta|^2}\eta,\xi-\eta\right\rangle\qquad\text{ for all }\  \xi,\eta\in\rn.
\end{equation}
Indeed, for superquadratic $G$ function $t\mapsto g(t)/t$ is nondecreasing, hence we have 
\begin{flalign*}
G(|\xi-\eta|)&\lesssim_{i_G} \tfrac{g(|\xi-\eta|)}{|\xi-\eta|}|\xi-\eta|^2 \lesssim_{i_G,s_G}\tfrac{g(|\xi|+|\eta|)}{|\xi|+|\eta|}|\xi-\eta|^2\approx_{i_G, s_G} \mathcal{J}_G(\xi,\eta).
\end{flalign*} 
Note that in the view of \eqref{inq-2-p-male} inequality~\eqref{eq:coro} cannot hold for subquadratic $G$.
\end{corollary} 

We point out that the claim of Theorem~\ref{theo:main} enables to control  a broader class of operators than just~\eqref{eq:A}. From the point of view of the growth and coercivity there is no reason for $\opA$ to have separated variables.
\begin{remark} 
Assume $G\in C(\rpp)\cap C^1((0,\infty))$ is an $N$-function such that $G\in\Delta_2\cap\nabla_2$. The natural family of second order differential operators generalizing $p$-Laplacian to operators of the form $-\dv\opA(x,Du)$, where  $\opA:\Omega\times\rn\to\rn$ is a Carath\'eodory function having the doubling Orlicz growth with respect to the second variable,  satisfies conditions
\begin{equation}
    \label{opA-cond}\begin{cases}
    |\opA(x,\xi)|\leq c_1^\opA g(|\xi|),\\\langle \opA(x,\xi)-\opA(x,\eta),\xi-\eta\rangle\geq c_2^\opA|\V(\xi)-\V(\eta)|^2\end{cases}
\end{equation}
with certain constants $c_1^\opA,c_2^\opA>0$ for all $ \xi,\eta\in\rn$. Let us note that the first line of~\eqref{opA-cond} implies that $\opA(x,\xi)=0$ whenever $\xi=0$. Moreover, the second line of~\eqref{opA-cond} for $\eta=0$  implies that  \[ c_2^\opA G(|\xi|)\leq \langle \opA(x,\xi),\xi\rangle.\]
 PDE problems under such regime are studied e.g. in~\cite{CYZG,CiMa}.
\end{remark}

{\bf Organization. } In Section~\ref{sec:prelim} we provide basic definitions and lemmas. In the subsequent two sections we provide results on the control over the monotonicity rates $\mathcal{J}_p$ from~\eqref{mono-p} and $\mathcal{J}_G$ from~\eqref{mono-G}. In fact, Section~\ref{sec:p} is devoted to  the precise analysis of the $p$-growth case including the discussion on the constants. Later on, in Section~\ref{sec:G}, we present related results for the Orlicz case summarized with the proof of Theorem~\ref{theo:main}.

\section{Preliminaries}\label{sec:prelim}

 

\begin{definition} [$N$-function and its conjugate]$ $\\
We say that a function $G:[0, \infty) \to [0, \infty)$ is an $N$-function if it is a convex, continuous,  and such that $G(0)=0$, $\lim_{t \to 0}{G(t)}/{t}=0$, and $\lim_{t \to \infty}{G(t)}/{t}=\infty$.\\
The Young conjugate $\wt{G}$  (called also the complementary function, or the Legendre transform) to an $N$-function $G:\rpp\to\rpp$  is given by the following formula
$\wt{G}(s):=\sup_{t>0}(s t-G(t)).$ 
\end{definition}

\begin{definition}[$\Delta_2$ and $\nabla_2$-condition]\label{def:D2} We say that a function $G:\rpp\to\rpp$ satisfies $\Delta_2$-condition (denoted $G\in\Delta_2$), if there exists $c_{\Delta_2}>0$ such that $G(2t)\leq c_{\Delta_2}G(t)$ for $t>0$.\\ We say that $G$ satisfy $\nabla_2$-condition if $\wt{G}\in\Delta_2.$  \end{definition}
\noindent 
The above conditions describe the speed and {the} regularity of the growth. Note that it is possible that $G$ satisfies only one of the conditions $\Delta_2/\nabla_2$. For instance, when $G(t) = (1+|t|)\log(1+|t|)-|t|$, its complementary function is  $\widetilde{G}(s)= \exp(|s|)-|s|-1$. Then $G\in\Delta_2$ and   $\widetilde{G}\not\in\Delta_2$.

\medskip

In the sequel we consider $G$ which is an $N$-function such that $G\in C(\rpp)\cap C^1((0,\infty))$ and we denote by $g$ its derivative ($g=G'$). Calling $\wt{G}$ complementary function is justified by the fact that
\begin{equation}
    \label{G-int-form}
G(t)=\int_0^t g(s)\,ds\qquad\text{and}\qquad \wt{G}(t)=\int_0^t g^{-1}(s)\,ds.
\end{equation}
For the proof see \cite[Section~8.3]{adamsfournier}. From this fact one can directly deduce the following fact.

\begin{lemma}\label{lem:1} If $G:\rpp\to\rpp$ is a $N$-function with a continuous derivative $g$, then
\[\tfrac{t}{2}g\left(\tfrac{t}{2}\right)\leq G(t)\leq tg(t).\]
\end{lemma}
\begin{proof}Due to the convexity of $G$, function $g$ is nondecreasing. In turn, the first inequality follows from the fact that \[G(t)=\int_0^tg(s)\,ds\geq\int_{t/2}^t g(s)\,ds\geq \tfrac{t}{2}g\left(\tfrac{t}{2}\right).\] The second inequality is also a consequence of monotonicity of $g$.
\end{proof}

Having the above fact, we infer the comparison of a doubling function with power-type functions.

\begin{lemma}\label{lem:iG-sG} Suppose $G\in C(\rpp)\cap C^1((0,\infty))$. Then $G\in\Delta_2\cap\nabla_2$ if and only if~\begin{equation}\label{iG-sG} 
1<  i_G=\inf_{t>0}\frac{tg(t)}{G(t)}\leq \sup_{t>0}\frac{tg(t)}{G(t)}=s_G<\infty.\end{equation}
Moreover, then
\begin{equation}\label{comp-i_G-s_G} 
\frac{G(t)}{t^{i_G}}\quad\text{is non-decreasing}\qquad\text{and}\qquad\frac{G(t)}{t^{s_G}}\quad\text{is non-increasing},
\end{equation} 
and \[i_GG(t)\leq g(t)t\leq s_GG(t).\]
\end{lemma}

\begin{remark}  Note that $G\not\in\Delta_2$ can be trapped between two power-type functions with arbitrarily close powers, see~\cite{CGZG,BDMS,KR} for various constructions. 
\end{remark}
We have the following consequences of Lemmas~\ref{lem:1} and~\ref{lem:iG-sG}.
\begin{corollary}\label{cor:g-delta2} If  $G\in\Delta_2\cap\nabla_2$, then  the constants in conditions $\Delta_2$ and $\nabla_2$ depend only on $i_G$ and $s_G$. Moreover, $g(t) \approx_{i_G, s_G} \tfrac{G(t)}{t} \approx_{i_G, s_G} \tfrac{G(2t)}{2t} \approx_{i_G, s_G} g(2t). $ 
\end{corollary}

\section{Operators of power growth} \label{sec:p}

Below we present detailed proofs of basic properties of the quantity~\eqref{mono-p}.   Lemmas~\ref{lem:p-male} and~\ref{lem:p>2} provide basic inequalities, for $1 < p \leq 2$ and $p > 2$ respectively. They are used to prove more general results presented by Lemmas~\ref{lem:Vp2},~\ref{lem:pow-approx} and~\ref{lem:Vp}, which are components of Corollary~\ref{coro:main}. Note that, in the view of counterexample of Remark~\ref{rem:ce}, the claim of Lemma~\ref{lem:p>2} differs from~\cite[Section 10, (I)]{Lindqvist}. 

\begin{lemma}[$1<p\leq 2$]\label{lem:p-male} Suppose $1<p\leq 2$.\\ Then there exists $c_1^p$ such that for all $\xi,\eta\in\rn$ such that $\xi\neq 0$ we have \begin{equation}
    \label{inq-1-p-male}\left\langle |\xi|^{p-2}\xi-|\eta|^{p-2}\eta,\xi-\eta\right\rangle\geq c_1^p\frac{|\xi-\eta|^2}{|\xi|^{2-p}+|\eta|^{2-p}}
\end{equation}
where the optimal constant is achieved when $\langle\xi,\eta\rangle =|\xi|\,|\eta|$ and is given by $c_1^p = \min\{1, 2(p - 1)\}$,
if $1 < p < 2$, and $c_1^2 = 2$.  Furthermore there exists $c_2^p>0$ such that for all $\xi,\eta\in\rn$ it holds that
\begin{equation}
\label{inq-2-p-male}\left\langle |\xi|^{p-2}\xi-|\eta|^{p-2}\eta,\xi-\eta\right\rangle \leq c_2^p |\xi-\eta|^p.
\end{equation}
\end{lemma}

\begin{proof}
We start with the proof of~\eqref{inq-1-p-male}   that is a corrected version of \cite[Lemma A.2]{BIV}. We notice first that if $p=2$ or $\eta=0$, there is nothing to prove. Note that for $1<p<2$, putting $\eta=0$ in~\eqref{inq-1-p-male} gives us that $c_1^p \leq 1$. Now, for $\eta\neq 0$ let us transform~\eqref{inq-1-p-male} to the following form
\[ \left( \tfrac{|\xi|^{2-p}}{|\eta|^{2-p}} + \tfrac{|\eta|^{2-p}}{|\xi|^{2-p}} + 2(1 - c^p_1) \right) \langle \xi, \eta \rangle \leq |\xi|^{2 - p}|\eta|^p + |\xi|^p|\eta|^{2 - p} + (1 - c^p_1)\left(|\xi|^2 + |\eta|^2\right). \]
Since $\langle \xi, \eta \rangle \leq |\xi||\eta|$ and the left-hand side of the last display is always positive for $c_1^p \leq 1$, then it suffices to prove the following inequality
\[ \left( \tfrac{|\xi|^{2-p}}{|\eta|^{2-p}} + \tfrac{|\eta|^{2-p}}{|\xi|^{2-p}} + 2(1 - c^p_1) \right) |\xi||\eta| \leq |\xi|^{2 - p}|\eta|^p + |\xi|^p|\eta|^{2 - p} + (1 - c^p_1)\left(|\xi|^2 + |\eta|^2\right), \]
which is equivalent to
\[ \tfrac{|\xi|^{2 - p}}{|\eta|^{2 - p}} + \tfrac{|\xi|^p}{|\eta|^p} + (1 - c_1^p) \left( \tfrac{|\xi|}{|\eta|} - 1 \right)^2 - \left(\tfrac{|\xi|^{2 - p}}{|\eta|^{2 - p}} + \tfrac{|\eta|^{2 - p}}{|\xi|^{2 - p}} \right) \tfrac{|\xi|}{|\eta|} \geq 0.\]

Without loss of generality we may assume that $|\xi| \geq |\eta|$ and then by letting $\lambda = \frac{|\xi|}{|\eta|}$, we need to prove that for $\lambda \geq 1$ the following holds true
\[ f_p(\lambda) := \lambda^{2 - p} + \lambda^p + (1 - c_1^p)(\lambda - 1)^2 - \lambda^{3 - p} - \lambda^{p-1} \geq 0 .\]
Observe that
\[ f_p(\lambda) = (\lambda-1)\big(\lambda^{p-1}-\lambda^{2-p}+(1-c_1^p)(\lambda-1)\big). \]

For $p \in [\frac{3}{2}, 2)$ we have $\lambda^{p-1} \geq \lambda^{2-p}$. Therefore, for every $c_1^p \in [0, 1]$, we have $f_p(\lambda) \geq 0$. The equality in~\eqref{inq-2-p-male} is achieved for $c_1^p = 1$ and $\xi \neq 0$, $\eta = 0$. It proves that this is the optimal constant. \newline

On the other hand, to show that $f_p(\lambda)\geq 0$ also for $p \in (1, \frac{3}{2})$, let us denote
\[ h_p(\lambda) := \lambda^{p-1}-\lambda^{2-p}+(1-c_1^p)(\lambda-1). \]
Note that for $\lambda \geq 1$, $f_p$ has the same sign as $h_p$. We can calculate
\begin{align*}
     h_p'(\lambda)& = (p-1)\lambda^{p-2} - (2-p)\lambda^{1-p} + (1-c_1^p), \\
 h_p''(\lambda) &= (2-p)(p-1)(\lambda^{-p}-\lambda^{p-3}). 
\end{align*}
Since $p < \frac{3}{2}$, we have $\lambda^{-p}-\lambda^{p-3} \geq 0$, which means that $h_p''(\lambda) \geq 0$.
Let $c_1^p = 2(p - 1)$. We have $h_p'(1) = 0$, which in conjunction with $h_p'' \geq 0$ gives us that $h_p' \geq 0$. Finally, $h_p(1) = 0$ implies that $h_p \geq 0$ and $f_p \geq 0$. The constant $c_1^p = 2(p - 1)$ is optimal, because for $\xi = a\eta$ for $a \to 1^{+}$ we have~\eqref{inq-1-p-male} equivalent to $c_1^p \leq 2(p -1)$. This completes the proof of~\eqref{inq-1-p-male}. \newline

To prove~\eqref{inq-2-p-male}, we follow the ideas of~\cite{Db2}. It suffices that there exists $c_2^p > 0$ such that
\begin{equation}
    \label{suff}\left| |\xi|^{p - 2}\xi - |\eta|^{p - 2}\eta \right| \leq c_2^p|\xi - \eta|^{p-1}.
\end{equation}
Indeed, then by the Cauchy--Schwarz inequality we have~\eqref{inq-2-p-male}. In order to show~\eqref{suff} we notice that by triangle inequality we have
\begin{align}
    \label{L=R1+R2}
L:=\left| |\xi|^{p - 2}\xi - |\eta|^{p - 2}\eta \right| \leq \big[|\eta|^{p - 2}|\xi - \eta|\big] + \left| |\xi|^{p-2} - |\eta|^{p-2} \right||\xi|. 
\end{align}

Observe that by the mean value theorem, there exists $t \in [0, 1]$ such that for the second term on the right-hand side in~\eqref{L=R1+R2} we can write
\begin{flalign*} 
\left| |\xi|^{p-2} - |\eta|^{p-2} \right||\xi|=& \left| \frac{\frac{1}{|\xi|^{2-p}} - \frac{1}{|\eta|^{2-p}}}{\frac{1}{|\xi|} - \frac{1}{|\eta|}} \right| \, \left| \frac{1}{|\xi|} - \frac{1}{|\eta|} \right| \, |\xi| = (2-p) \left( \frac{t}{|\xi|} + \frac{1 - t}{|\eta|} \right)^{1-p} \, \left| |\eta| - |\xi| \right| \, \frac{1}{|\eta|} \\
&\leq (2-p) \left( \frac{(1 - t)|\xi| + t|\eta|}{|\xi| |\eta|} \right)^{1-p} \, |\xi - \eta| \, \frac{1}{|\eta|}.\end{flalign*}
Since due to concavity of $s\mapsto s^{p-1}$ we have
\[t|\eta|^{p-1}+(1-t)|\xi|^{p-1}\leq (t|\eta|+(1-t)|\xi|)^{p-1},\]
we can  estimate
\begin{flalign*}
\left| |\xi|^{p-2} - |\eta|^{p-2} \right||\xi|&\leq (2-p)\frac{|\eta|^{p-2}|\xi|^{p-1}|\xi - \eta|}{t|\eta|^{p-1} + (1-t)|\xi|^{p-1}}\\
&= \frac{(2-p)|\xi|^{p-1}}{t|\eta|^{p-1} + (1-t)|\xi|^{p-1}}\left[ |\eta|^{p - 2}|\xi - \eta|\right],
\end{flalign*}
where in the last square brackets in the last display we have the first term of the right-hand side in~\eqref{L=R1+R2}.
Therefore for some  $t \in [0, 1]$ we have
\[ L=\left| |\xi|^{p - 2}\xi - |\eta|^{p - 2}\eta \right| \leq \left[ |\eta|^{p - 2}|\xi - \eta|\right] \left( 1 + \frac{(2-p)|\xi|^{p-1}}{t|\eta|^{p-1} + (1-t)|\xi|^{p-1}} \right) .\]
and, by analogy,  also
\[L \leq\left[ |\xi|^{p-2}|\xi - \eta|\right] \left( 1 + \frac{(2-p)|\eta|^{p-1}}{t'|\xi|^{p-1} + (1-t')|\eta|^{p-1}} \right)\quad \text{for some $t' \in [0, 1]$.} \]
 
Let us consider two cases: $|\eta| > \frac{1}{2}|\xi - \eta|$ and $|\eta| \leq \frac{1}{2}|\xi - \eta|$.

If $|\eta| > \frac{1}{2}|\xi - \eta|$, then the following inequalities hold true
\[ |\eta|^{p-2} \leq 2^{2-p}|\xi - \eta|^{p-2} \qquad\text{and} \qquad |\eta| > \tfrac{1}{3}|\xi|, \]
which imply that
\begin{flalign}\nonumber L\leq |\eta|^{p - 2}|\xi - \eta| \left( 1 + \frac{(2 - p)|\xi|^{p-1}}{t|\eta|^{p-1} + (1-t)|\xi|^{p-1}} \right) &\leq 2^{2 - p}\left( 1 +  \frac{2 - p}{3^{1-p}t + 1 - t} \right)|\xi - \eta|^{p-1}  \\
& \leq 2^{2 - p}(1 + (2-p)3^{p - 1})|\xi - \eta|^{p-1}. \label{L1}
\end{flalign}

If $|\eta| \leq  \frac{1}{2}|\xi - \eta|$, then there holds
\[ |\eta| \leq \tfrac12|\xi-\eta| \leq \tfrac12(|\xi|+|\eta|) \leq |\xi|, \]
which enable us to estimate
\[ \frac{|\eta|^{p-1}}{t'|\xi|^{p-1} + (1 - t')|\eta|^{p - 1}} \leq 1. \]
In this case we also have
\[ |\xi|^{p-2}|\xi - \eta| = |\xi - \eta|^{p-1}\left(\tfrac{|\xi - \eta|}{|\xi|} \right)^{2-p} \leq 2^{2-p}|\xi - \eta|^{p-1}. \]
Altogether
\begin{flalign} L \leq |\xi|^{p-2}|\xi - \eta| \left( 1 + \frac{(2-p)|\eta|^{p-1}}{t'|\xi|^{p-1} + (1-t')|\eta|^{p-1}} \right)\leq 2^{3-p} |\xi-\eta|^{p-1}.\label{L2}
\end{flalign}

Summing it up, by~\eqref{L1} and~\eqref{L2} there exists $c_2^p > 0$ such that~\eqref{suff} holds true from which~\eqref{inq-2-p-male} follows.
\end{proof}

The following facts are proven in 
\cite{Lindqvist}. We correct a flaw in the proof of the second inequality, see Remark~\ref{rem:ce} for details.
\begin{lemma}[$2\leq p<\infty$]\label{lem:p>2} If $2\leq p<\infty$, then  for all $\xi,\eta\in\rn$ we have 
\begin{equation} \label{eq:p>3} \left\langle |\xi|^{p-2}\xi-|\eta|^{p-2}\eta,\xi-\eta\right\rangle \geq \tfrac{1}{2}|\xi-\eta|^2( |\xi|^{p-2}+|\eta|^{p-2}) \end{equation}
and
\begin{equation} \label{eq:3>p>2} \left\langle |\xi|^{p-2}\xi-|\eta|^{p-2}\eta,\xi-\eta\right\rangle \geq C_p|\xi-\eta|^p, \end{equation}
where $C_p=2^{2-p}$ is optimal.\end{lemma}
\begin{proof} Firstly, we prove~\eqref{eq:p>3} following~\cite{Lindqvist}. Let us start with the identity
\begin{equation*}
\left\langle |\xi|^{p-2}\xi-|\eta|^{p-2}\eta,\xi-\eta\right\rangle = \frac{|\xi|^{p-2} + |\eta|^{p-2}}{2}|\xi - \eta|^2 + \frac{(|\xi|^{p-2} - |\eta|^{p-2})(|\xi|^2 - |\eta|^2)}{2}  .
\end{equation*}
As the second component on the right-hand side is always non-negative, we have
\begin{equation*}
\left\langle |\xi|^{p-2}\xi-|\eta|^{p-2}\eta,\xi-\eta\right\rangle \geq \frac{|\xi|^{p-2} + |\eta|^{p-2}}{2}|\xi - \eta|^2,
\end{equation*}
which is the first inequality of our claim. \newline

Now we shall prove~\eqref{eq:3>p>2}. Let us consider two cases associated with the sign of $|\eta|^2 - \langle \xi, \eta \rangle$. Firstly, we assume that $|\eta|^2 - \langle \xi, \eta \rangle \leq 0$. We denote
\[ f(p) := \left\langle |\xi|^{p-2}\xi-|\eta|^{p-2}\eta,\xi-\eta\right\rangle - 2^{2-p}|\xi-\eta|^p. \]
We need to prove that $f \geq 0$. Without loss of generality, by the scaling argument, we may assume that $|\xi| = 1 \geq |\eta|$. Observe that
\[ f'(p) = |\eta|^{p-2}(|\eta|^2-\langle \xi, \eta \rangle)\ln{|\eta|} + 2^{2-p}|\xi-\eta|^p\ln{\frac{2}{|\xi-\eta|}}. \]
Since $|\eta| \leq 1 = |\xi|$, we have $\ln|\eta| \leq 0$ and $\ln{|\xi-\eta|} \leq \ln{(1 + |\eta|)} \leq \ln{2}$, which along with $|\eta|^2 - \langle \xi, \eta \rangle \leq 0$ gives us that $f' \geq 0$. It is easy to see that $f(2) = 0$, which means that $f \geq 0$ and~\eqref{eq:3>p>2} is proved for $|\eta|^2 - \langle \xi, \eta \rangle \leq 0$. \newline

Let us assume that $|\eta|^2 - \langle \xi, \eta \rangle \geq 0$ and restrict to $|\xi| \geq |\eta|$ again. We can transform the left-hand side of~\eqref{eq:3>p>2} in the following way
\[\left \langle |\xi|^{p-2}\xi-|\eta|^{p-2}\eta, \xi - \eta \right\rangle = |\xi|^p + |\eta|^p - \langle \xi, \eta \rangle \left(|\xi|^{p-2}+|\eta|^{p-2}\right) = |\xi|^{p-1}\left(|\xi|-\tfrac{\langle \xi, \eta \rangle}{|\xi|}\right)+|\eta|^{p-1}\left(|\eta|-\tfrac{\langle \xi, \eta \rangle}{|\eta|}\right) \]
\[ = \left(|\xi|+|\eta|-\langle \xi, \eta \rangle\left(\tfrac{1}{|\xi|} + \tfrac{1}{|\eta|}\right) \right)\left( |\xi|^{p-1}\frac{|\xi|-\tfrac{\langle \xi, \eta \rangle}{|\xi|}}{|\xi|+|\eta|-\langle \xi, \eta \rangle\left(\tfrac{1}{|\xi|} + \tfrac{1}{|\eta|}\right)} + |\eta|^{p-1}\frac{|\eta|-\tfrac{\langle \xi, \eta \rangle}{|\eta|}}{|\xi|+|\eta|-\langle \xi, \eta \rangle\left(\tfrac{1}{|\xi|} + \tfrac{1}{|\eta|}\right)} \right). \]
Because of the convexity of $t \mapsto t^{p-1}$, we can use Jensen's inequality to estimate
\[\left \langle |\xi|^{p-2}\xi-|\eta|^{p-2}\eta, \xi - \eta\right \rangle \geq \left(|\xi|+|\eta|-\langle \xi, \eta \rangle\left(\tfrac{1}{|\xi|} + \tfrac{1}{|\eta|}\right) \right)^{2-p}|\xi-\eta|^{2p-2}. \]
Therefore, it is sufficient to prove that
\[ \left(|\xi|+|\eta|-\langle \xi, \eta \rangle\left(\tfrac{1}{|\xi|} + \tfrac{1}{|\eta|}\right) \right)^{2-p}|\xi-\eta|^{2p-2} \geq 2^{2-p}|\xi-\eta|^p, \]
which can be transformed into
\[ |\xi|+|\eta|-\langle \xi, \eta \rangle\left(\tfrac{1}{|\xi|} + \tfrac{1}{|\eta|}\right) \leq 2|\xi-\eta| \]
and then into
\[ (|\xi| + |\eta|)\left(1 - \tfrac{\langle \xi, \eta \rangle}{|\xi||\eta|}\right) \leq 2|\xi-\eta|, \]
which is equivalent to
\begin{equation} \label{eq:sierpien2} (|\xi|+|\eta|)^2\left(1 - \tfrac{\langle \xi, \eta \rangle}{|\xi||\eta|}\right)^2 \leq 4|\xi-\eta|^2. \end{equation}
Let $\lambda = \tfrac{\langle \xi, \eta \rangle}{|\xi||\eta|}$ and observe that
\[ 4|\xi-\eta|^2 - (|\xi|+|\eta|)^2\left(1 - \tfrac{\langle \xi, \eta \rangle}{|\xi||\eta|}\right)^2 = 4\left(|\xi|^2-2\lambda|\xi||\eta| + |\eta|^2\right) - (|\xi|+|\eta|)^2(1-\lambda)^2. \]
Since $\langle \xi, \eta \rangle \leq |\xi||\eta|$, it is easy to see that minimal value of the expression above is achieved for $\lambda=-1$ or $\lambda=1$. For $\lambda = 1$, it is equal to $4(|\xi| - |\eta|)^2 \geq 0$. For $\lambda=-1$ it is equal to $0$. It means that~\eqref{eq:sierpien2} holds true. Therefore,~\eqref{eq:3>p>2} is proved for $|\eta|^2-\langle \xi, \eta \rangle \geq 0$, which is the last case needed. Note that constant $C_p = 2^{2-p}$ is optimal. Indeed, equality in~\eqref{eq:3>p>2} is achieved if $\xi = -\eta$.
\end{proof}

\begin{remark}\label{rem:ce}
In~\cite{Lindqvist}, it is claimed that~\eqref{eq:p>3} is always stronger than~\eqref{eq:3>p>2}, i.e. 
\[ \tfrac12|\xi-\eta|^2\left(|\xi|^{p-2} + |\eta|^{p-2}\right) \geq 2^{2-p}|\xi-\eta|^p. \]
It holds true only for $p \geq 3$. Indeed, for such $p$ we can use convexity of $t \mapsto t^{p-2}$ to estimate
\[ \frac{|\xi|^{p-2}+|\eta|^{p-2}}{2}|\xi-\eta|^2 \geq 2^{2-p}(|\xi|+|\eta|)^{p-2}|\xi-\eta|^2\geq 2^{2-p}|\xi-\eta|^p. \]
However, for $p \in [2, 3)$ it does not hold true for every $\xi, \eta \in \R^n$. Indeed, for $\eta = 0, \xi \neq 0$ we have
\[ 2^{2-p}|\xi-\eta|^p = 2^{2-p}|\xi|^p > \tfrac12|\xi|^p = \tfrac12|\xi-\eta|^2(|\xi|^{p-2}+|\eta|^{p-2}). \]
\end{remark}
 
Let us recall function $\bar V_p^\mu$ defined in \eqref{Vp}. The following relation holds true.
\begin{lemma}\label{lem:Vp2}
If $0 < p < \infty$, then for all $\mu \geq 0$ 
\[\frac{|\bar V_p^\mu(\xi)-\bar V_p^\mu(\eta)|^2}{|\xi-\eta|^2}\approx_p \left(\mu+|\xi|^2+|\eta|^2\right)^\frac{p-2}{2}.\]
\end{lemma}

\begin{proof}
In the case $p < 2$ we follow the ideas of~\cite{hamburger}. Let us define function $W_p^\mu$
\[ W_p^\mu(t) = (\mu + t^2)^{\frac{p - 2}{4}}t. \]
It satisfies $|\bar{V}_p^\mu(x)| = W_p^\mu(|x|)$. Observe that 
\[ (W_p^\mu)'(t) = (\mu + t^2)^\frac{p - 6}{4}\left(\mu + \tfrac{pt^2}{2} \right). \]
Since $p>0$, $(W_p^\mu)'(t)$ is positive for $t \geq 0$, and consequently $W_p^\mu$ is strictly increasing. Also, we have
\begin{equation} \label{eq:sierpien0}
 \tfrac{p}{2}(\mu + t^2)^\frac{p - 2}{4} \leq (W_p^\mu)'(t) \leq (\mu + t^2)^\frac{p - 2}{4}.
 \end{equation}
Let us assume that $|\xi| \geq |\eta|$ and let $\lambda = \frac{\langle \xi, \eta \rangle}{|\xi||\eta|}$. Using this notation we have
\[ \frac{|\bar V_p^\mu(\xi)-\bar V_p^\mu(\eta)|^2}{|\xi-\eta|^2(\mu+|\xi|^2+|\eta|^2)^\frac{p-2}{2}} = \frac{W_p^\mu(|\eta|)^2 + W_p^\mu(|\xi|)^2 - 2W_p^\mu(|\eta|)W_p^\mu(|\xi|)\lambda}{(|\xi|^2 + |\eta|^2 - 2|\xi||\eta|\lambda)(\mu+|\xi|^2+|\eta|^2)^\frac{p-2}{2}}=:F_{p,\mu,\xi,\eta}(\lambda). \] 
We will justify that $F_{p,\mu,\xi,\eta}$ is bounded below and above on $[-1,1]$. Observe that
\[ F_{p,\mu,\xi,\eta}'(\lambda) = \frac{2|\xi||\eta|(W_p^\mu(|\eta|)^2 + W_p^\mu(|\xi|)^2) - 2W_p^\mu(|\eta|)W_p^\mu(|\xi|)(|\xi|^2+|\eta|^2) }{(|\xi|^2+|\eta|^2-2|\xi||\eta|\lambda)^2(\mu + |\xi|^2 + |\eta|^2)^{p-2}}. \]
The denominator of the expression above is always positive, while the numerator is independent of $\lambda$, which implies that $F_{p,\mu,\xi,\eta}'$ has constant sign. It gives us that  extremal values of $F_{p,\mu,\xi,\eta}$ are achieved for $\lambda = 1$ and $\lambda = -1$. Therefore, it is sufficient to prove that there exist constants $c, C > 0$ such that
\begin{equation} \label{eq:sierpien4} c \leq
\sqrt{F_{p,\mu,\xi,\eta}(-1)}, \sqrt{F_{p,\mu,\xi,\eta}(1)}
\leq C. \end{equation}
In order to estimate $ \sqrt{F_{p,\mu,\xi,\eta}(-1)}$ we see that
\begin{align*}
1& = \frac{(\mu + |\xi|^2 + |\eta|^2)^\frac{p-2}{4}|\xi| + (\mu + |\xi|^2 + |\eta|^2)^\frac{p-2}{4}|\eta|}{(\mu + |\xi|^2 + |\eta|^2)^\frac{p-2}{4}(|\xi| + |\eta|)} \leq \frac{W_p^\mu(|\xi|) + W_p^\mu(|\eta|)}{(\mu + |\xi|^2 + |\eta|^2)^\frac{p-2}{4}(|\xi| + |\eta|)}\\ & = \sqrt{F_{p,\mu,\xi,\eta}(-1)}\leq \frac{W_p^\mu(|\xi|) + W_p^\mu(|\eta|)}{W_p^\mu(|\xi| + |\eta|)} \leq 2, \end{align*}
where the last inequality follows from the monotonicity of $W_p^\mu$.

Now we shall estimate $\sqrt{F_{p,\mu,\xi,\eta}(1)}$. By applying the mean value theorem, there exists $r \in [|\eta|, |\xi|]$ such that
\[  \sqrt{F_{p,\mu,\xi,\eta}(1)}=\frac{W_p^\mu(|\xi|) - W_p^\mu(|\eta|)}{(\mu + |\xi|^2 + |\eta|^2)^\frac{p - 2}{4}(|\xi| - |\eta|)} = \frac{(W_p^\mu)'(r)}{(\mu + |\xi|^2 + |\eta|^2)^\frac{p - 2}{4}}\,. \] 
Due to~\eqref{eq:sierpien0} we can estimate
\[ \frac{(W_p^\mu)'(r)}{(\mu + |\xi|^2 + |\eta|^2)^\frac{p - 2}{4}} \geq \frac{p}{2} \frac{(\mu + r^2)^\frac{p - 2}{4}}{(\mu + |\xi|^2 + |\eta|^2)^\frac{p - 2}{4}} \geq \frac{p}{2}\,, \]
which is sufficient lower bound. To obtain upper bound, we consider two cases. 

If $3|\eta| \leq |\xi|$, then $|\xi| - |\eta| \geq \frac{1}{2}(|\xi| + |\eta|)$ and due to the monotonicity of $W_p^\mu$, we obtain
\[ \frac{W_p^\mu(|\xi|) - W_p^\mu(|\eta|)}{(\mu + |\xi|^2 + |\eta|^2)^\frac{p - 2}{4}(|\xi| - |\eta|)} \leq \frac{W_p^\mu(|\xi|)}{(\mu + (|\xi| + |\eta|)^2)^\frac{p - 2}{4}\frac{1}{2}(|\xi| + |\eta|)} = \frac{2W_p^\mu(|\xi|)}{W_p^\mu(|\xi| + |\eta|)} \leq 2.\]
On the other hand, if $3|\eta| > |\xi|$ then $|\eta| < r < |\xi|$ implies that $|\xi| < 3r$ and we have $\mu + |\xi|^2 + |\eta|^2 < 10(\mu + r^2)$. Taking into account~\eqref{eq:sierpien0}, we infer that
\[ \frac{W_p^\mu(|\xi|) - W_p^\mu(|\eta|)}{(\mu + |\xi|^2 + |\eta|^2)^\frac{p - 2}{4}(|\xi| - |\eta|)} = \frac{(W_p^\mu)'(r)}{(\mu + |\xi|^2 + |\eta|^2)^\frac{p - 2}{4}} \leq \frac{(\mu + r^2)^\frac{p - 2}{4}}{(\mu + |\xi|^2 + |\eta|^2)^\frac{p - 2}{4}} \leq 10^{\frac{p-2}{4}} \leq 10^\frac{1}{2}. \]
Summing up all cases completes the proof of~\eqref{eq:sierpien4} and therefore, proof of the lemma for $p < 2$.

\bigskip

In the case $p \geq 2$ we follow the ideas of~\cite{giaquinta}.  By direct computation one can verify that
\begin{flalign*} |\bar V_p^\mu(\xi) -\bar V_p^\mu(\eta)| &\leq \int_0^1 \left| \frac{d}{dt}\bar V_p^\mu(t\xi + (1-t)\eta) \right|dt \leq \tfrac{p}{2}|\xi - \eta|  \int_0^1 \left(\mu + |t\xi + (1-t)\eta|^2\right)^\frac{p-2}{4}dt\,\\
& \leq 2^\frac{p-6}{4}p|\xi-\eta|\left(\mu + |\xi|^2 + |\eta|^2\right)^\frac{p-2}{4}. \end{flalign*}
In conclusion, we have one of two inequalities we need to prove. For the reverse inequality, without loss of generality we assume that $|\xi| \geq |\eta|$. We distinguish two cases.  

If $|\xi| \geq 2|\eta|$, we have $|\xi - \eta| \leq \frac{3}{2}|\xi|$ and
\[ \left(\mu + |\xi|^2 + |\eta|^2\right)^\frac{p - 2}{4}|\xi - \eta| \leq 2^\frac{p - 2}{4}\left(\mu + |\xi|^2\right)^\frac{p - 2}{4} \cdot \tfrac{3}{2}|\xi| = 3 \cdot 2^\frac{p-6}{4}|\bar V_p^\mu(\xi)|. \]
Since $|\bar V_p^\mu(\xi)| = W_p^\mu(|\xi|)$ is increasing function of $|\xi|$, we have
\[ |\bar V_p^\mu(\xi) - \bar V_p^\mu(\eta)| \geq |\bar V_p^\mu(\xi)| - |\bar V_p^\mu(\eta)| \geq |\bar V_p^\mu(\xi)| - |\bar V_p^\mu(\tfrac{1}{2}\xi)| \geq \tfrac{1}{2}|\bar V_p^\mu(\xi)|, \]
which is sufficient according to the previous inequality. \newline

If $|\xi| < 2|\eta|$, we have $|t\xi - \eta| \geq |\xi - \eta|$ for every $t \geq 1$. Let us set
\[ L_p^\mu(x) = \left(\mu + |x|^2\right)^\frac{p - 2}{4}.\]
We have
\[ |\bar V_p^\mu(\xi) - \bar V_p^\mu(\eta)| = L_p^\mu(\eta) \left| \tfrac{L_p^\mu(\xi)}{L_p^\mu(\eta)}\xi - \eta \right| \geq L_p^\mu(\eta)|\xi - \eta| \geq \left(\mu + |\xi|^2 + |\eta|^2\right)^\frac{p - 2}{4}|\xi - \eta|5^{\frac{2-p}{4}},\]
which is the last inequality needed to prove the lemma.
\end{proof}
We use Lemma~\ref{lem:Vp2} to prove the following fact.
\begin{lemma}
\label{lem:pow-approx}
If $1<p<\infty$, then 
    \begin{equation}
\label{eq:lem:pow-approx}
\left\langle |\xi|^{p-2}\xi-|\eta|^{p-2}\eta,\xi-\eta\right\rangle\approx_p \big(|\xi|+|\eta|\big)^{p-2}|\xi-\eta|^2.    \end{equation} 
Note that this is impossible for $p\leq 1$. \end{lemma}

\begin{proof} If $p=2$, then there is nothing to prove as \eqref{eq:lem:pow-approx} reads $|\xi-\eta|^2\approx |\xi-\eta|^2$. Now we prove the lemma for $p > 2$. Let us start with the identity used in the proof of Lemma~\ref{lem:p>2}
\[\left\langle |\xi|^{p-2}\xi-|\eta|^{p-2}\eta,\xi-\eta\right\rangle = \frac{|\xi|^{p-2} + |\eta|^{p-2}}{2}|\xi - \eta|^2 + \frac{\left(|\xi|^{p-2} - |\eta|^{p-2}\right)\left(|\xi|^2 - |\eta|^2\right)}{2}. \]
Since the last summand of the equality above is always non-negative and we have $|\xi|^{p-2}+|\eta|^{p-2} \approx_p \left(|\xi| + |\eta|\right)^{p-2}$, it is sufficient to prove
\begin{equation} \label{eq:sierpien3} \left(|\xi|^{p-2} - |\eta|^{p-2}\right)\left(|\xi|^2 - |\eta|^2\right) \lesssim_p \left(|\xi|^{p-2}+|\eta|^{p-2}\right)|\xi-\eta|^2. \end{equation}
Without loss of generality we may assume that $|\xi| \geq |\eta|$. Note that for $x \geq 1$ we have
\[ (x^{p-2}-1)(x+1) \lesssim_p (x^{p-2} + 1)(x-1). \]
Indeed, it is immediate consequence of the fact that the following limits are finite and positive
\[ \lim_{x\to1^+} \tfrac{(x^{p-2}-1)(x+1)}{(x^{p-2} + 1)(x-1)} = 2p - 4 ,\quad \lim_{x\to+\infty} \tfrac{(x^{p-2}-1)(x+1)}{(x^{p-2} + 1)(x-1)} = 1.\]
Putting $x = \tfrac{|\xi|}{|\eta|}$ and multiplying by $|\eta|^{p-1}$ give us that
\[ \left(|\xi|^{p-2} - |\eta|^{p-2}\right)(|\xi| + |\eta|) \lesssim_p \left(|\xi|^{p-2}+|\eta|^{p-2}\right)(|\xi| - |\eta|). \]
Therefore
\[ \left(|\xi|^{p-2} - |\eta|^{p-2}\right)\left(|\xi|^2 - |\eta|^2\right) \lesssim_p \left(|\xi|^{p-2}+|\eta|^{p-2}\right)\left(|\xi| - |\eta|\right)^2 \leq \left(|\xi|^{p-2}+|\eta|^{p-2}\right)|\xi - \eta|^2, \]

which proves~\eqref{eq:sierpien3} and ends proof of the lemma for $p \geq 2$. \newline

For $1<p < 2$ let us use~\eqref{inq-1-p-male} to state that
\[  c_1^p \frac{|\xi - \eta|^2}{|\xi|^{2-p} + |\eta|^{2-p}} \leq \left\langle |\xi|^{p-2}\xi-|\eta|^{p-2}\eta,\xi-\eta\right\rangle.\]
By the concavity of $t \mapsto t^{2-p}$, we have
\[ \frac{1}{|\xi|^{2-p} + |\eta|^{2-p}} \geq 2^{1-p}\big(|\xi|+|\eta|\big)^{p-2}, \]
which is enough to state that
\[ \big(|\xi|+|\eta|\big)^{p-2}|\xi-\eta|^2 \lesssim_p \left\langle |\xi|^{p-2}\xi-|\eta|^{p-2}\eta,\xi-\eta\right\rangle.  \]
Moreover, we have
\begin{equation} \label{eq:wrzesien0}
\big(|\xi|^2+|\eta|^2\big)^\frac{p-2}{2} \approx_p \big(|\xi| + |\eta| \big)^{p-2}.  
\end{equation}
Indeed
\[ \tfrac{|\xi| + |\eta|}{2} \leq \sqrt{\tfrac{|\xi|^2 + |\eta|^2}{2}} \leq \sqrt{\tfrac{|\xi|^2 + 2|\xi||\eta|+ |\eta|^2}{2}} = \tfrac{|\xi| + |\eta|}{\sqrt{2}},\] which raised to the power  $p-2$ immediately gives~\eqref{eq:wrzesien0}.
Therefore, by Lemma~\ref{lem:Vp2} with $\mu = 0$ and $\bar p = 2p - 2$ and~\eqref{eq:wrzesien0}, we have
\[ \big| |\xi|^{p-2}\xi - |\eta|^{p-2}\eta \big| \lesssim_p |\xi - \eta| \left(|\xi|^2 + |\eta|^2\right)^\frac{p-2}{2} \lesssim_p |\xi - \eta|(|\xi| + |\eta|)^{p-2}. \]
By the Cauchy--Schwarz inequality we get
\[\left \langle |\xi|^{p-2}\xi-|\eta|^{p-2}\eta,\xi-\eta\right\rangle \lesssim_p |\xi - \eta|^2(|\xi| + |\eta|)^{p-2}, \]
which is the last needed inequality.\newline

If $p=1$, then the left-hand side of~\eqref{eq:lem:pow-approx} vanishes for $\eta=c\xi$, $c\neq0$, but the right-hand side is positive in this case. 
If $p < 1$, then for $\xi = c\eta$ and $c \to 0^+$, the left-hand side of~\eqref{eq:lem:pow-approx} is $|\xi|^p(1-c)(1-c^{p-1})$ and converges to $\infty$ and the right-side therein is $|\xi|^p(1-c)^2(1+c)^{p-2}$, which converges to $|\xi|^p$. \end{proof}
The last two lemmas used together give the following result.
\begin{lemma}\label{lem:Vp}
If $1\leq p<\infty$, then  for all $\xi,\eta\in\rn$ we have
\[\left\langle |\xi|^{p-2}\xi-|\eta|^{p-2}\eta,\xi-\eta\right\rangle\approx_p |V_p(\xi)-V_p(\eta)|^2.\]
\end{lemma}
 
\begin{proof}
Since by Lemma~\ref{lem:Vp2} we have 
\[ |V_p(\xi)-V_p(\eta)|^2 \approx_p \big(|\xi|^2 + |\eta|^2\big)^{\frac{p-2}{2}}|\xi-\eta|^2\]
and by Lemma~\ref{lem:pow-approx}
we have \[\left\langle |\xi|^{p-2}\xi-|\eta|^{p-2}\eta,\xi-\eta\right\rangle \approx_p \big(|\xi| + |\eta| \big)^{p-2}|\xi-\eta|^2, \]
equivalence \eqref{eq:wrzesien0} completes the proof. \end{proof}
If one is interested in precise values of a constant in Lemma~\ref{lem:Vp}, we provide an easy proof for $p\geq 2$ due to~\cite{Lindqvist}.
\begin{remark}\label{rem:Vp}
If $2\leq p<\infty$, then for all $\xi, \eta \in \R^N$ we have
\[\langle |\xi|^{p-2}\xi-|\eta|^{p-2}\eta,\xi-\eta\rangle\geq \tfrac{4}{p^2}|V_p(\xi)-V_p(\eta)|^2.\]
\end{remark}
\begin{proof}
 Note that
\begin{flalign*}
 |\xi|^{p - 2}\xi - |\eta|^{p -2}\eta &= \int_{0}^{1} \frac{d}{dt}|\eta + t(\xi - \eta)|^{p-2}(\eta + t(\xi - \eta))dt \\ &= (\xi - \eta)\int_{0}^{1} |\eta + t(\xi - \eta)|^{p - 2}dt\\
 &\quad+ (p - 2)\int_{0}^{1} |\eta + t(\xi - \eta)|^{p - 4}\langle \eta + t(\xi - \eta), \xi - \eta \rangle (\eta + t(\xi - \eta))dt. \end{flalign*}
By applying it the left-hand side of our claim, we have
\begin{flalign}\nonumber\langle |\xi|^{p-2}\xi-|\eta|^{p-2}\eta,\xi-\eta \rangle  &= |\xi - \eta|^2\int_{0}^{1} |\eta + t(\xi - \eta)|^{p - 2}dt\\
&\quad + (p - 2)\int_{0}^{1} |\eta + t(\xi - \eta)|^{p - 4}\big(\langle \eta + t(\xi - \eta), \xi - \eta \rangle\big)^2 dt.\label{split} \end{flalign}
By the Cauchy--Schwarz inequality we have following inequality  
\[ 0 \leq \int_{0}^{1} |\eta + t(\xi - \eta)|^{p - 4}\big(\langle \eta + t(\xi - \eta), \xi - \eta \rangle\big)^2 dt \leq |\xi - \eta|^2\int_{0}^{1} |\eta + t(\xi - \eta)|^{p - 2}dt. \]
Thus, by~\eqref{split} we may write that
\[ |\xi - \eta|^2\int_{0}^{1} |\eta + t(\xi - \eta)|^{p - 2}dt \leq \langle |\xi|^{p-2}\xi-|\eta|^{p-2}\eta,\xi-\eta \rangle \leq (p-1)|\xi - \eta|^2\int_{0}^{1} |\eta + t(\xi - \eta)|^{p - 2}dt. \]
Note that again by the Cauchy--Schwarz inequality and the mean value theorem we infer that 
\[ |\xi - \eta| \int_{0}^{1} |\eta + t(\xi - \eta)|^{p - 2}dt \leq \big||\xi|^{p-2}\xi-|\eta|^{p-2}\eta\big| \leq (p-1) |\xi - \eta| \int_{0}^{1} |\eta + t(\xi - \eta)|^{p - 2}dt. \]
Replacing $p$ by $\frac{p + 2}{2}\geq 2$ in the first inequality in the last display we get
\begin{flalign*}
 |V_p(\xi)-V_p(\eta)|^2 & = \big||\xi|^{\frac{p - 2}{2}}\xi - |\eta|^{\frac{p - 2}{2}}\eta\big|^2 \leq \frac{p^2}{4}|\xi - \eta|^2 \left( \int_{0}^{1} |\eta + t(\xi - \eta)|^{\frac{p-2}{2}}dt \right)^2 \\
& \leq \frac{p^2}{4}|\xi - \eta|^2  \int_{0}^{1} |\eta + t(\xi - \eta)|^{p-2}dt \leq \frac{p^2}{4} \langle |\xi|^{p-2}\xi-|\eta|^{p-2}\eta,\xi-\eta \rangle,
\end{flalign*} 
which completes the proof.
\end{proof}

\section{General growth}\label{sec:G}
In this section we concentrate on controlling the monotonicity quantity $\mathcal{J}_G$ defined in~\eqref{mono-G} related to the study of the operator from~\eqref{eq:A} involving an $N$-function $G\in\Delta_2\cap\nabla_2$, which is continuously differentiable, and its derivative $g=G'$. 
 One may observe analogy between Lemmas~\ref{lem:Vp2} and~\ref{lem:VG} and between Lemmas~\ref{lem:pow-approx} and~\ref{lem:VG2}. We present much more elementary proofs than those accessible in literature.

In the next proofs we will make use of the following simple fact.
\begin{lemma}\label{lem:t-p}
For every $p > 0$ and $t \geq 1$, we have $1 - t^{-p} \leq p(t-1)$.
\end{lemma}

The following lemma yields the result related to~\cite[Lemma 20]{DiEt}, but the proof is essentially more elementary and does not require $G$ to be twice differentiable. 
\begin{lemma} Suppose $G\in\Delta_2\cap \nabla_2$, then for all $ \xi,\eta\in\rn$ we have \label{lem:VG}  \begin{equation*}
\frac{g(|\xi|+|\eta|)}{|\xi|+|\eta|} \approx_{i_G, s_G} \frac{|\V(\xi)-\V(\eta)|^2}{|\xi-\eta|^2}.
\end{equation*}
\end{lemma} 
\begin{proof} 
    Observe that when in definition of $V_p$, i.e. \eqref{Vp}, one takes $p=1$ and a vector $G(|\xi|)\tfrac{\xi}{|\xi|}$, it holds that that\begin{equation}
    \label{VG=V1}
    \V(\xi) = V_1\left(\frac{G(|\xi|)}{|\xi|}\xi\right)\qquad\text{for every $\xi \in \R^n$}.
\end{equation}  
Hence,    Lemma~\ref{lem:Vp2} gives us that
    \begin{align*}
    |\V(\xi) - \V(\eta)|^2 &= \big|V_1\left(\tfrac{G(|\xi|)}{|\xi|}\xi\right) - V_1\left(\tfrac{G(|\eta|)}{|\eta|}\eta\right) \big|^2 \\
    & \approx \big| \tfrac{G(|\xi|)}{|\xi|}\xi - \tfrac{G(|\eta|)}{|\eta|}\eta \big|^2 \left( G(|\xi|)^2 + G(|\eta|)^2 \right)^{-\frac{1}{2}}. 
    \end{align*}
    Without loss of generality, we may assume that $|\xi| \geq |\eta|$, which gives us that
    \[ |\V(\xi) - \V(\eta)|^2 \approx \big| \tfrac{G(|\xi|)}{|\xi|}\xi - \tfrac{G(|\eta|)}{|\eta|}\eta \big|^2 \cdot \tfrac{1}{G(|\xi|)}. \]
    Moreover, by Lemma~\ref{lem:iG-sG} and Corollary~\ref{cor:g-delta2} we have
    \[ \tfrac{g(|\xi| + |\eta|)}{|\xi| + |\eta|} \approx_{i_G, s_G} \tfrac{g(|\xi|)}{|\xi|} \approx_{i_G, s_G} \tfrac{G(|\xi|)}{|\xi|^2}, \]
    therefore it suffices now to show that 
    \[ \big| \tfrac{G(|\xi|)}{|\xi|}\xi - \tfrac{G(|\eta|)}{|\eta|}\eta \big|^2 \tfrac{1}{G(|\xi|)} 
    \approx_{i_G, s_G} \tfrac{G(|\xi|)}{|\xi|^2}|\xi - \eta|^2, \]
    which can be simplified into
    \begin{align}
        \label{fin}
    \big| \tfrac{G(|\xi|)}{|\xi|}\xi - \tfrac{G(|\eta|)}{|\eta|}\eta \big| \approx_{i_G, s_G} \tfrac{G(|\xi|)}{|\xi|}|\xi - \eta|. 
    \end{align}
    
    Now we shall prove inequality `$\lesssim$'  in~\eqref{fin}. Let us consider two cases.\\
    Assume $|\eta|\leq |\xi| \leq 2|\eta|$. For every $t \geq 1$, $|t\xi - \eta| \geq |\xi - \eta|$, which by monotonicity of $s \mapsto \frac{G(s)}{s}$ implies that
    \[ \left| \frac{\frac{G(|\xi|)}{|\xi|}}{\frac{G(|\eta|)}{|\eta|}}\xi - \eta \right| \geq |\xi - \eta| \geq \tfrac{1}{c_{\Delta_2}} \tfrac{|\eta|}{|\xi|} \tfrac{G(|\xi|)}{G(|\eta|)}|\xi-\eta|, \]
    which is sufficient for our inequality. \\
    If $|\xi| > 2|\eta|$, then by triangle inequality we have
    \begin{align*}
         \left| \tfrac{G(|\xi|)}{|\xi|}\xi - \tfrac{G(|\eta|)}{|\eta|}\eta \right|& \geq \tfrac{G(|\xi|)}{|\xi|}  |\xi| - \tfrac{G(|\eta|)}{|\eta|}  |\eta|\geq \tfrac{G(|\xi|)}{|\xi|}  \tfrac{|\xi|}{2} \geq \tfrac{G(|\xi|)}{|\xi|} \tfrac{|\xi - \eta|}{4}. 
    \end{align*}
    The first  inequality needed for~\eqref{fin} is proved. 
    
    In order to show $\gtrsim$ in~\eqref{fin}, observe that for $|\xi|\geq |\eta|$ it holds
    \[ \left| \xi - \frac{\tfrac{G(|\eta|}{|\eta|}}{\tfrac{G(|\xi|)}{|\xi|}}\eta \right| \leq |\xi - \eta| + |\eta| \left( 1 - \frac{\tfrac{G(|\eta|}{|\eta|}}{\frac{G(|\xi|)}{|\xi|}}\right).  \]
    Moreover, in order to estimate the last term we might use that $|\xi| - |\eta|\leq |\xi - \eta| $, which means that this is sufficient to prove the following 
    \begin{equation}
        \label{est}
    1 - \frac{\tfrac{G(|\eta|)}{|\eta|}}{\tfrac{G(|\xi|)}{|\xi|}} \lesssim_{i_G, s_G} \tfrac{|\xi|}{|\eta|} - 1. 
    \end{equation}
    Indeed, then \[ \left| \xi - \frac{\tfrac{G(|\eta|}{|\eta|}}{\tfrac{G(|\xi|)}{|\xi|}}\eta \right| \leq |\xi - \eta| + c(i_G, s_G)(|\xi|-|\eta|)\lesssim_{i_G, s_G}  |\xi - \eta|.  \]
    In order to show~\eqref{est} we use Lemma~\ref{lem:t-p} together with the fact that $G\in\Delta_2$, which gives us
    \[ 1 - \frac{\tfrac{G(|\eta|)}{|\eta|}}{\tfrac{G(|\xi|)}{|\xi|}} \leq 1 - \left(\tfrac{|\eta|}{|\xi|} \right)^{s_G - 1} \leq (s_G - 1)\left( \tfrac{|\xi|}{|\eta|} - 1 \right). \]
    This ends proof of the second inequality needed for~\eqref{fin}, and hence the proof of the lemma is complete. 
    
\end{proof} 
By Lemma~\ref{lem:VG} we can conclude the following fact.
\begin{corollary}
 Suppose  $G\in\Delta_2\cap \nabla_2$, then for all $ \xi,\eta\in\rn$ we have
    \label{lem:g-VsqrtG}\[g(|\xi|+|\eta|)|\xi-\eta|\approx_{i_G, s_G} G^\frac{1}{2}(|\xi|+|\eta|) |\V(\xi)-\V(\eta)|.\]
\end{corollary} 
\begin{proof}
    From Lemma~\ref{lem:VG} we have \[g^\frac{1}{2}(|\xi|+|\eta|)|\xi-\eta|\approx_{i_G, s_G} |\V(\xi)-\V(\eta)|(|\xi|+|\eta|)^\frac{1}{2},\]
    and therefore
    \[g(|\xi|+|\eta|)|\xi-\eta|\approx_{i_G, s_G} |\V(\xi)-\V(\eta)|\left((|\xi|+|\eta|)g(|\xi|+|\eta|)\right)^\frac{1}{2}.\]
    The claim follows from Lemma~\ref{lem:iG-sG}.
\end{proof}
The following lemma presents the result analogous to Lemma~\ref{lem:pow-approx}.
\begin{lemma}
Suppose $G\in\Delta_2\cap \nabla_2$, then for all $ \xi,\eta\in\rn$ we have  \label{lem:VG2}
\[\left\langle \tfrac{G(|\xi|)}{|\xi|^2}\xi-\tfrac{G(|\eta|)}{|\eta|^2}\eta,\xi-\eta\right\rangle\approx_{i_G, s_G} \tfrac{g(|\xi|+|\eta|)}{|\xi|+|\eta|}|\xi-\eta|^2.\]    
\end{lemma}

\begin{proof}We denote\[I_1:= \left\langle \tfrac{G(|\xi|)}{|\xi|^2}\xi-\tfrac{G(|\eta|)}{|\eta|^2}\eta,\xi-\eta\right\rangle. \]

    Let us start with showing inequality `$\gtrsim$'. By Lemma~\ref{lem:VG} it is sufficient to prove that
    \[I_1\gtrsim_{i_G, s_G} |\V(\xi)-\V(\eta)|^2\approx_{i_G, s_G} \tfrac{g(|\xi|+|\eta|)}{|\xi|+|\eta|}|\xi-\eta|^2.\]
    
     Noting that $\V(\xi) = V_1(\frac{G(|\xi|)}{|\xi|}\xi)$ and using Lemma~\ref{lem:Vp} for $p=1$ and vectors $\frac{G(|\xi|)}{|\xi|}\xi$ and $\frac{G(|\eta|)}{|\eta|}\eta$, we may write
    \[ I_2:= \left\langle \tfrac{\xi}{|\xi|} - \tfrac{\eta}{|\eta|}, \tfrac{G(|\xi|)}{|\xi|}\xi - \tfrac{G(|\eta|)}{|\eta|}\eta \right\rangle \gtrsim_{i_G, s_G}|\V(\xi)-\V(\eta)|^2.\]
    Therefore, it is enough to prove that $I_2\lesssim I_1$. Notice that
    \[I_2= G(|\xi|) + G(|\eta|) - \langle \xi, \eta \rangle  \tfrac{G(|\xi|) + G(|\eta|)}{|\xi| \, |\eta|} \]
    and, on the other hand, we have
    \[ I_1=
    G(|\xi|) + G(|\eta|) - \langle \xi, \eta \rangle \left( \tfrac{G(|\xi|)}{|\xi|^2} + \tfrac{G(|\eta|)}{|\eta|^2} \right). \]
    Without loss of generality, we may assume that $|\eta| \leq |\xi|$. We know that $\tfrac{G(|\eta|)}{|\eta|} \leq \tfrac{G(|\xi|)}{|\xi|},$     which by multiplication by $\tfrac{|\xi| - |\eta|}{|\xi|\,|\eta|}$ and rearranging terms leads to
    %
    %
    %
    %
    %
    \[ \tfrac{G(|\xi|)}{|\xi|^2} + \tfrac{G(|\eta|)}{|\eta|^2} \leq \tfrac{G(|\xi|) + G(|\eta|)}{|\xi| \, |\eta|}. \]
    Therefore, if $\langle \xi, \eta \rangle \geq 0$, then we immediately have $I_2\leq I_1$. 
    %
    %
    On the other hand, for $\langle \xi, \eta \rangle < 0$, observe that
    \begin{align*}
         I_2 &\leq 2(G(|\xi|) + G(|\eta|)) \\
         &\leq 2\left( G(|\xi|) + G(|\eta|) - \langle \xi, \eta \rangle   \left( \tfrac{G(|\xi|)}{|\xi|^2} + \tfrac{G(|\eta|)}{|\eta|^2} \right)\right)=2I_1. 
    \end{align*}
    Summing up, we have  $|\V(\xi)-\V(\eta)|^2 \lesssim_{i_G, s_G} I_2 \lesssim I_1$ and the first inequality of the claim is proven.\newline
    
    To complete the proof, it remains to show inequality `$\lesssim$', that is 
    \begin{equation}
        \label{lipiec1}
    \left\langle \tfrac{G(|\xi|)}{|\xi|^2}\xi-\tfrac{G(|\eta|)}{|\eta|^2}\eta,\xi-\eta\right\rangle \lesssim_{i_G, s_G} \tfrac{g(|\xi|+|\eta|)}{|\xi|+|\eta|}|\xi-\eta|^2, 
    \end{equation}
    which we prove for $|\eta| \leq |\xi|$.     By the Cauchy--Schwarz inequality, we have
    \[ I_1 \leq \left| \tfrac{G(|\xi|)}{|\xi|^2}\xi-\tfrac{G(|\eta|)}{|\eta|^2}\eta \right| |\xi - \eta|. \]
    Moreover, since $|\xi|+|\eta|\approx|\xi|$ and by Lemma~\ref{lem:iG-sG}, it holds that
    \[\tfrac{g(|\xi|+|\eta|)}{|\xi|+|\eta|}|\xi-\eta|^2\approx_{i_G, s_G} \tfrac{G(|\xi|)}{|\xi|^2}|\xi-\eta|^2. \]
    Therefore, to have~\eqref{lipiec1}  it is sufficient to prove that
    \[ \left| \tfrac{G(|\xi|)}{|\xi|^2}\xi-\tfrac{G(|\eta|)}{|\eta|^2}\eta \right| = |\xi - \eta| \lesssim_{i_G, s_G} \tfrac{G(|\xi|)}{|\xi|^2}|\xi-\eta|^2, \]
    which is equivalent to
    \begin{equation}
        \label{lipiec2}
    \left| \xi-\frac{\tfrac{G(|\eta|)}{|\eta|^2}}{\tfrac{G(|\xi|)}{|\xi|^2}}\eta \right| \lesssim_{i_G, s_G} |\xi-\eta|. 
    \end{equation}
    We will prove it in two cases: when  ${\frac{G(|\eta|)}{|\eta|^2}} \geq {\frac{G(|\xi|)}{|\xi|^2}}$ and  ${\frac{G(|\eta|)}{|\eta|^2}} < {\frac{G(|\xi|)}{|\xi|^2}}$. By triangle inequality it holds
    \[ \left| \xi-\frac{\frac{G(|\eta|)}{|\eta|^2}}{\frac{G(|\xi|)}{|\xi|^2}}\eta \right| \leq |\xi - \eta| + |\eta|  \left| 1 - \frac{\frac{G(|\eta|)}{|\eta|^2}}{\frac{G(|\xi|)}{|\xi|^2}} \right| = |\xi - \eta| + |\eta|\left( \frac{\frac{G(|\eta|)}{|\eta|^2}}{\frac{G(|\xi|)}{|\xi|^2}} - 1 \right)=:II. \]
    
    Suppose that ${\frac{G(|\eta|)}{|\eta|^2}} \geq {\frac{G(|\xi|)}{|\xi|^2}}$. In order to estimate the right-hand side in the last display we use the fact that  $\tfrac{G(|\eta|)}{G(|\xi|)} \leq \tfrac{|\eta|}{|\xi|}$, which leads to
    \[ \frac{\frac{G(|\eta|)}{|\eta|^2}}{\frac{G(|\xi|)}{|\xi|^2}} - 1 \leq \frac{|\xi|}{|\eta|} - 1.\]
    Then
    \[ II\leq |\xi - \eta| + |\eta|\left(\tfrac{|\xi|}{|\eta|} - 1\right) \leq 2|\xi-\eta| \]
    and~\eqref{lipiec2} holds true.  
    
    Let us assume the opposite, i.e. ${\frac{G(|\eta|)}{|\eta|^2}} <\frac{G(|\xi|)}{|\xi|^2}$. By Lemma~\ref{lem:iG-sG} we have 
    \[ 1 - \frac{\frac{G(|\eta|)}{|\eta|^2}}{\frac{G(|\xi|)}{|\xi|^2}} \leq 1 - \left( \frac{|\eta|}{|\xi|}\right)^{s_G - 2}. \]
    We assume that the left-hand side above is positive, therefore $s_G - 2 > 0$. In turn, by Lemma~\ref{lem:t-p}, we have
    \[ 1 - \frac{\frac{G(|\eta|)}{|\eta|^2}}{\frac{G(|\xi|)}{|\xi|^2}} \leq 1 - \left( \frac{|\eta|}{|\xi|}\right)^{s_G - 2} \leq (s_G - 2) \left( \frac{|\xi|}{|\eta|} - 1\right) \]
    and we can estimate 
    \[II \leq |\xi-\eta| + (s_G-2)|\eta|\left(\tfrac{|\xi|}{|\eta|} - 1\right) \lesssim_{s_G} |\xi-\eta|.\]
    It ends proof of~\eqref{lipiec2}. Therefore, `$\lesssim$' inequality of the claim is proved. 
\end{proof}

We are in a position to prove our main result, that is Theorem~\ref{theo:main}, by summarizing the above facts.

\begin{proof}[Proof of Theorem~\ref{theo:main}] Let us recall $G_a$ defined in~\eqref{G-shift}.  We will prove now a following relation\begin{equation}
    \label{eq:1line}\tfrac{\G(\xi,\eta)}{|\xi-\eta|^2}\approx_{i_G, s_G}\tfrac{G_{|\xi|}(|\xi-\eta|)}{|\xi-\eta|^2}\approx_{i_G, s_G} \tfrac{g(|\xi|+|\eta|)}{|\xi|+|\eta|}.
\end{equation}Observe that the derivative of $G_a$  is
\begin{equation}\label{ga}
g_a(s) = \frac{g(a+s)}{a+s}s.
\end{equation}
Given definition of $\G$, i.e.~\eqref{cG}, and the fact that $G_a(t)\approx_{i_G, s_G} g_a(t)t$, which holds due to Lemma~\ref{lem:iG-sG}, we infer that \begin{equation}\label{Ga}
    G_a(t)\approx_{i_G, s_G}\begin{cases}G(t),&t\geq a,\\
    \tfrac{G(a)}{a^2}t^2,& t\leq a,
    \end{cases}=\G(a,t).
\end{equation} 
Using~\eqref{ga}, we have
\begin{flalign}\label{equiv3}
\tfrac{G_{|\xi|}(|\xi-\eta|)}{|\xi-\eta|^2}&\approx_{i_G, s_G} \tfrac{g_{|\xi|}(|\xi - \eta|)}{|\xi - \eta|} =  \tfrac{g(|\xi| + |\xi - \eta|)}{|\xi| + |\xi - \eta|} \approx_{i_G, s_G} \tfrac{G(|\xi| + |\xi - \eta|)}{(|\xi| + |\xi - \eta|)^2}.
\end{flalign}
Since $G\in\Delta_2\cap \nabla_2$ and $|\xi| + |\xi - \eta| \approx |\xi| + |\eta|,$ we have that
\begin{flalign*}
\tfrac{G_{|\xi|}(|\xi-\eta|)}{|\xi-\eta|^2}&\approx_{i_G, s_G} \tfrac{G(|\xi| + |\eta|)}{(|\xi| + |\eta|)^2} \approx_{i_G, s_G} \tfrac{g(|\xi| + |\eta|)}{|\xi| + |\eta|},
\end{flalign*}
which ends proof of~\eqref{eq:1line}. Then \eqref{eq:1line} in conjunction with Lemma~\ref{lem:VG} immediately gives us~\eqref{equiv-theo:main}.
\end{proof}



\end{document}